\documentclass[a4paper]{amsart}
\numberwithin{equation}{section}
\newtheorem{theorem}{Theorem}[section]
\newtheorem{corollary}{Corollary}[section]
\newtheorem{definition}{Definition}[section]
\newtheorem{lemma}{Lemma}[section]

\theoremstyle{remark}
\newtheorem{remark}{Remark}[section]
\usepackage{amsmath, hyperref, amsfonts, amstext, amsgen, amsbsy, amsopn}
\usepackage{color}
\usepackage{graphicx}
\usepackage{subfigure}
\usepackage[margin=3cm]{geometry}
\title[Coefficient and Fekete-Szeg\"{o} problem estimates]
 {Coefficient and Fekete-Szeg\"{o} problem estimates for certain subclass of analytic and bi-univalent functions}
\begin{document}
\begin{abstract}
In this paper, we obtain the Fekete-Szeg\"{o} problem for the $k$-th $(k\geq1)$ root transform of the analytic and normalized functions $f$ satisfying the condition
\begin{equation*}
    1+\frac{\alpha-\pi}{2 \sin \alpha}<
    {\rm Re}\left\{\frac{zf'(z)}{f(z)}\right\} <
    1+\frac{\alpha}{2\sin \alpha} \quad (|z|<1),
\end{equation*}
where $\pi/2\leq \alpha<\pi$. Afterwards, by the above two-sided inequality we introduce and investigate a certain subclass of analytic and bi-univalent functions in the disk $|z|<1$ and obtain upper bounds for the first few coefficients and Fekete-Szeg\"{o} problem for functions belonging to this analytic and bi-univalent function class.
\end{abstract}
\author[H. Mahzoon]
       {H. Mahzoon}

\address{Department of Mathematics, Islamic Azad University, Firoozkouh
Branch, Firoozkouh, Iran}
\email {hesammahzoon1@gmail.com, mahzoon$_{-}$hesam@yahoo.com {\rm (H. Mahzoon)}}

\maketitle
\section{Introduction}
Let $\mathcal{A}$ be the class of functions $f$ of the form
\begin{equation}\label{f}
f(z)=z+ \sum_{n=2}^{\infty}a_{n}z^{n}=z+a_2z^2+a_3z^3+\cdots,
\end{equation}
which are analytic in the open unit disk $\Delta=\{z\in \mathbb{C} : |z|<1\}$ and normalized by the condition $f(0)=f'(0)-1=0$. Also let $\mathcal{P}$ be the class of functions $p$ analytic in $\Delta$ of the form
\begin{equation*}
  p(z)=1+p_1z+p_2z^2+\cdots+p_n z^n+\cdots,
\end{equation*}
such that ${\rm Re}\{p(z)\}>0$ for all $z\in\Delta$.
The subclass of all functions in $\mathcal{A}$ which are univalent (one-to-one) in $\Delta$ is denoted by $\mathcal{S}$. A well-known example for the class $\mathcal{S}$ is the {\it Koebe} function which has the following form
\begin{equation*}
  k(z):=\frac{z}{(1-z)^2}=z+2z^2+3z^3+\cdots+nz^n+\cdots\quad(z\in\Delta).
\end{equation*}
It is known that the Koebe function maps the open unit disk $\Delta$ onto the entire plane minus the interval $(-\infty,-1/4]$.
Also, the well-known {\it Koebe One-Quarter Theorem} states that the image of the open unit disk $\Delta$ under every function $f\in\mathcal{S}$ contains the disk $\{w:|w|<\frac{1}{4}\}$, see \cite[Theorem 2.3]{Duren}. Therefore, according to the above, every function $f$ in the class $\mathcal{S}$ has an inverse $f^{-1}$ which satisfies the following conditions:
\begin{equation*}
  f^{-1}(f(z))=z\quad(z\in\Delta)
\end{equation*}
and
\begin{equation*}
  f(f^{-1}(w))=w\quad(|w|<r_0(f);\ \ r_0(f)\geq1/4),
\end{equation*}
where
\begin{equation}\label{f inverse}
  f^{-1}(w)=w-a_2w^2+(2a_2^2-a_3)w^3-(5a_2^3-5a_2a_3+a_4)w^4+\cdots=:g(w).
\end{equation}
We say that a function $f\in\mathcal{A}$ is {\it bi-univalent} in $\Delta$ if, and only if, both $f$ and $f^{-1}$ are univalent in $\Delta$. We denote by $\Sigma$ the class of all bi-univalent functions in $\Delta$. The following functions
\begin{equation*}
  \frac{z}{1-z},\quad -\log(1-z)\quad {\rm and} \quad \frac{1}{2}\log\left(\frac{1+z}{1-z}\right),
\end{equation*}
with the corresponding inverse functions, respectively,
\begin{equation*}
  \frac{w}{1+w},\quad \frac{\exp(w)-1}{\exp(w)}\quad {\rm and} \quad\frac{\exp(2w)-1}{\exp(2w)+1},
\end{equation*}
belong to the class $\Sigma$.
It is clear that the Koebe function is not a member of the class $\Sigma$, also the following functions
\begin{equation*}
  z-\frac{1}{2}z^2 \quad {\rm and} \quad \frac{z}{1-z^2}.
\end{equation*}
The initial coefficients estimate of the class of bi-univalent functions $\Sigma$ is studied by Lewin in 1967 and he obtained the bound $1.51$ for the modulus of the second coefficient $|a_2|$, see \cite{Lewin1967}. Afterward, Brannan and Clunie conjectured that $|a_2| \leq \sqrt{2}$, see \cite{BC}. Finally, in 1969, Netanyahu \cite{Netanyahu} showed that $\max_{f\in\Sigma} |a_2|=4/3$. For the another coefficients $a_n$ $(n\geq3)$ the sharp estimate is presumably still an open problem.

Let $f$ and $g$ be two analytic functions in $\Delta$. We say that a function $f$ is subordinate to $g$, written as
\begin{equation*}
  f(z)\prec g(z)\quad{\rm or}\quad f\prec g,
\end{equation*}
if there exists a Schwarz function $w:\Delta\rightarrow\Delta$ with
the following properties
\begin{equation*}
  w(0)=0\quad{\rm and}\quad |w(z)|<1\quad(z\in\Delta),
\end{equation*}
such that $f(z)=g(w(z))$ for all $z\in\Delta$. In particular, if
$g\in\mathcal{S}$, then we have the following geometric equivalence
relation
\begin{equation*}
  f(z)\prec g(z)\Leftrightarrow f(0)=g(0)\quad{\rm and}\quad f(\Delta)\subset g(\Delta).
\end{equation*}

A function $f\in\mathcal{A}$ is called starlike (with respect to $0$) if $tw\in f(\Delta)$ whenever $w\in f(\Delta)$ and $t\in[0, 1]$. We denote by $\mathcal{S}^*$ the class of all starlike functions in $\Delta$. Also, we say that a function $f\in\mathcal{A}$ is starlike of order $\gamma$ ($0\leq\gamma<1$) if, and only if,
\begin{equation*}
  {\rm Re}\left\{\frac{zf'(z)}{f(z)}\right\}>\gamma\quad(z\in\Delta).
\end{equation*}
The class of the starlike functions of order $\gamma$ in $\Delta$ is denoted by $ \mathcal{S}^*(\gamma)$. As usual we put $\mathcal{S}^*(0)\equiv \mathcal{S}^*$.

We recall that a function $f\in\mathcal{A}$ belongs
to the class $\mathcal{M}(\alpha)$ if $f$ satisfies the following two-sided inequality
\begin{equation*}
    1+\frac{\alpha-\pi}{2 \sin \alpha}<
    {\rm Re}\left\{\frac{zf'(z)}{f(z)}\right\} <
    1+\frac{\alpha}{2\sin \alpha} \quad (z\in\Delta),
\end{equation*}
where $\pi/2\leq \alpha<\pi$. The class $\mathcal{M}(\alpha)$ was introduced by Kargar {\it et al.} in \cite{KES}. We define the function $\phi$ as follows
\begin{equation*}
  \phi(\alpha):=1+\frac{\alpha-\pi}{2 \sin \alpha}\quad(\pi/2\leq \alpha<\pi).
\end{equation*}
Since
\begin{equation*}
  2\phi'(\alpha)=[(\pi-\alpha)\cot \alpha+1]\csc \alpha\quad(\pi/2\leq \alpha<\pi),
\end{equation*}
therefore for each $\alpha\in[\pi/2,\pi)$ we see that $\phi'(\alpha)\neq0$.
On the other hand, since $\phi(\pi/2)=1-\pi/4\approx 0.2146$ and
\begin{equation*}
  \lim_{\alpha\rightarrow \pi^-}\phi(\alpha)=\frac{1}{2},
\end{equation*}
thus the class $\mathcal{M}(\alpha)$ is a subclass of the starlike functions of order $\gamma$ where $0.2146\leq \gamma<0.5$. By this fact that $\mathcal{S}^*(\gamma)\subset\mathcal{S}$ for each $\gamma\in[0,1)$, thus we conclude that the members of the class $\mathcal{M}(\alpha)$ are univalent in $\Delta$.

Now, we recall the following result for the class $\mathcal{M}(\alpha)$, see \cite[Lemma 1.1]{KES}.
\begin{lemma}\label{lemKES}
Let $f(z)\in \mathcal{A}$ and $\pi/2\leq \alpha<\pi$. Then $f\in\mathcal{M}(\alpha)$ if, and only if,
\begin{equation*}
\left(\frac{z f'(z)}{f(z)}-1\right)\prec \mathcal{B}_\alpha(z)
\quad (z\in\Delta),
\end{equation*}
where
\begin{equation}\label{B alpha}
    \mathcal{B}_\alpha(z):=\frac{1}{2i\sin
    \alpha}\log\left(\frac{1+ze^{i\alpha}}{1+ze^{-i\alpha}}\right)\quad
    (z\in \Delta).
\end{equation}
\end{lemma}
The function $\mathcal{B}_\alpha(z)$ is convex univalent and has the form
\begin{equation}\label{sumB}
    \mathcal{B}_\alpha(z)=\sum_{n=1}^{\infty}A_n z^n \quad (z\in
    \Delta),
\end{equation}
where
\begin{equation*}
    A_n:= \frac{(-1)^{(n-1)}\sin n\alpha}{n \sin \alpha}
    \quad (n=1,2,\ldots).
\end{equation*}
Also we have $\mathcal{B}_\alpha(\Delta)=\Omega_\alpha$ where
\begin{equation*}
  \Omega_\alpha:=\left\{\zeta\in\mathbb{C}:\frac{\alpha-\pi}{2 \sin \alpha}<
    {\rm Re}\left\{\zeta\right\} <
    \frac{\alpha}{2\sin \alpha}, \ \ \frac{\pi}{2}\leq \alpha<\pi\right\}.
\end{equation*}

Very recently Sun {\it et al.} (see \cite{sun2019}) and Kwon and Sim (see \cite{kwonsim}) have studied the class $\mathcal{M}(\alpha)$. Sun {\it et al.} showed if the function $f$ of the form \eqref{f} belongs to the class $\mathcal{M}(\alpha)$, then $|a_n|\leq1$ while the estimate is not sharp. Subsequently, Kwon and Sim obtained sharp estimates on the initial coefficients $a_2$, $a_3$, $a_4$ and $a_5$ of the functions $f$ belonging to the class $\mathcal{M}(\alpha)$. The coefficient estimate problem for each of the Taylor-Maclaurin coefficients $|a_n|$ $(n=6,7,\ldots)$ is still an open question. Also, the logarithmic coefficients of the function $f\in\mathcal{M}(\alpha)$ were estimated by Kargar, see \cite{kargarJAnal}.

It is interesting to mention this subject that Brannan and Taha \cite{BT1986} introduced certain subclass of the bi-univalent function class $\Sigma$, denoted by $\mathcal{S}^*_\Sigma(\gamma)$ similar to the class of the starlike functions of order $\gamma$ $(0\leq\gamma<1)$. For each function $f\in\mathcal{S}^*_\Sigma(\gamma)$ they found non-sharp estimates for the initial
Taylor-Maclaurin coefficients. Recently, motivated by the Brannan and Taha's work, many authors investigated the coefficient bounds for various subclasses of the
bi-univalent function class $\Sigma$, see for instance \cite{Bulut, sharma2017, sri2018, sri2010, sri2013, sri2017, sri2015, sri2016}.

In this paper, motivated by the aforementioned works, we introduce and investigate a certain subclass of $\Sigma$ similar to the class $\mathcal{M}(\alpha)$ as follows.
\begin{definition}\label{Def M Singma alpha}
Let $\pi/2\leq \alpha<\pi$.
  A function $f\in\Sigma$ is said to be in the class $\mathcal{M}_\Sigma(\alpha)$, if the following inequalities hold:
\begin{equation*}
    1+\frac{\alpha-\pi}{2 \sin \alpha}<
    {\rm Re}\left\{\frac{zf'(z)}{f(z)}\right\} <
    1+\frac{\alpha}{2\sin \alpha} \quad (z\in\Delta)
\end{equation*}
and
\begin{equation*}
    1+\frac{\alpha-\pi}{2 \sin \alpha}<
    {\rm Re}\left\{\frac{wg'(w)}{g(w)}\right\} <
    1+\frac{\alpha}{2\sin \alpha} \quad (w\in\Delta),
\end{equation*}
where $g$ is defined by \eqref{f inverse}.
\end{definition}
\begin{remark}
Upon letting $\alpha\rightarrow\pi^-$ it is readily seen that a
function $f\in\Sigma$ is in the class $\mathcal{M}_\Sigma(1/2)$ if the following inequalities are satisfied:
\begin{equation*}
    {\rm Re}\left\{\frac{zf'(z)}{f(z)}\right\} >\frac{1}{2}
     \quad (z\in\Delta)
\end{equation*}
and
\begin{equation*}
    {\rm Re}\left\{\frac{wg'(w)}{g(w)}\right\} >\frac{1}{2}
    \quad (w\in\Delta),
\end{equation*}
where $g$ is defined by \eqref{f inverse}.
\end{remark}

The following lemma will be useful.
\begin{lemma}\label{lem. Fek-Sze}{\rm(}see \cite{Ma-M}{\rm)}
  Let the function $p$ be of the form belongs to the class $\mathcal{P}$. Then
for any complex number $\mu$ we have
\begin{equation*}
  \left|p_2-\mu p_1^2\right|\leq\left\{
                                  \begin{array}{ll}
                                    -4\mu+2, & \hbox{if\ \ $\mu\leq0$;} \\
                                    2, & \hbox{if\ \ $0\leq\mu\leq1$;} \\
                                    4\mu-2, & \hbox{if\ \ $\mu\geq1$.}
                                  \end{array}
                                \right.
\end{equation*}
The result is sharp for the cases $\mu<0$ or $\mu>1$ if and only if $p(z)=\frac{1+z}{1-z}$ or one of its rotations. If $0<\mu<1$, then the equality
holds if and only if $p(z)=\frac{1+z^2}{1-z^2}$ or one of its rotations. For the case $\mu=0$, the equality holds if and only if
\begin{equation*}
  p(z)=\frac{1}{2}(1+\nu)\frac{1+z}{1-z}+\frac{1}{2}(1-\nu)\frac{1-z}{1+z}\quad(0\leq \nu\leq1),
\end{equation*}
or one of its rotations. If $\mu=1$, the equality holds if and only if
\begin{equation*}
  \frac{1}{p(z)}=\frac{1}{2}(1+\nu)\frac{1+z}{1-z}+\frac{1}{2}(1-\nu)\frac{1-z}{1+z}
\quad(0\leq \nu\leq1),
\end{equation*}
or one of its rotations.
\end{lemma}
This paper is organized as follows. In Section \ref{sec2} we derive the Fekete-Szeg\"{o} coefficient functional associated with the $k$-th root transform for functions in the class $\mathcal{M}(\alpha)$. In Section \ref{sec3} we propose to find the estimates on the Taylor-Maclaurin coefficients $|a_2|$, $|a_3|$ and Fekete-Szeg\"{o} problem for functions in the class $\mathcal{M}_\Sigma(\alpha)$ which we introduced in Definition \ref{Def M Singma alpha}.
\section{Fekete-Szeg\"{o} problem for the class $\mathcal{M}(\alpha)$}\label{sec2}
Recently, many authors have obtained the Fekete-Szeg\"{o} coefficient functional associated with the $k$-th root transform for certain subclasses of analytic functions, see for instance \cite{ali2009, kesBooth, KMN}. In this section, we investigate this problem for the class $\mathcal{M}(\alpha)$.
At first, we recall that for a univalent function $f$ of the form \eqref{f}, the $k$-th root
transform is defined by
\begin{equation}\label{F(z)}
 F_k(z):=(f(z^k))^{1/k}=z+\sum_{n=1}^{\infty}b_{kn+1}z^{kn+1}\quad
 (z\in \Delta, k\geq1).
\end{equation}
For $f$ given by \eqref{f}, we have
\begin{equation}\label{FF(z)}
 (f(z^{k}))^{1/k}=z+\frac{1}{k}a_2z^{k+1}
 +\left(\frac{1}{k}a_3-\frac{1}{2}\frac{k-1}{k^2}a_2^2\right)z^{2k+1}+\cdots.
\end{equation}
Equating the coefficients of \eqref{F(z)} and \eqref{FF(z)} yields
\begin{equation}\label{bk+1 b2k+1}
 b_{k+1}=\frac{1}{k}a_2\quad {\rm and}\quad b_{2k+1}=\frac{1}{k}a_3-\frac{1}{2}\frac{k-1}{k^2}a_2^2.
\end{equation}
Now we have the following.
\begin{theorem}\label{t3.1}
Let $\pi/2\leq \alpha<\pi$ and $f\in\mathcal{M}(\alpha)$. If $F$ is the
$k$-th $(k\geq1)$ root transform of the function $f$ defined by \eqref{F(z)}, then for any
complex number $\mu$ we have
\begin{equation}\label{1t31}
 \left|b_{2k+1}-\mu
 b_{k+1}^2\right|\leq\left\{
                       \begin{array}{ll}
                         \frac{1}{2k}\left(1-\cos\alpha-\frac{2\mu+k-1}{k}\right), & \hbox{if\ \ $\mu\leq \delta_1$;} \\\\
                         \frac{1}{2k}, & \hbox{if\ \ $\delta_1\leq\mu\leq\delta_2$;} \\\\
                         \frac{1}{2k}\left(\cos\alpha+\frac{2\mu+k-1}{k}-1\right), & \hbox{$if\ \ \mu\geq \delta_2$,}
                       \end{array}
                     \right.
\end{equation}
where $\delta_1:=(1-k(1+\cos\alpha))/2$, $\delta_2:=(1+k(1-\cos\alpha))/2$ and $b_{2k+1}$ and $b_{k+1}$ are defined by \eqref{bk+1 b2k+1}. The result is sharp.
\end{theorem}
\begin{proof}
Let $\pi/2\leq \alpha<\pi$. If $f\in\mathcal{M}(\alpha)$, then by Lemma \ref{lemKES} and by definition of subordination, there exists a Schwarz function $w$ such that
\begin{equation}\label{2t31}
 \frac{zf'(z)}{f(z)}=1+\mathcal{B}_{\alpha}(w(z))\quad(z\in\Delta),
\end{equation}
where $\mathcal{B}_{\alpha}$ is defined by \eqref{B alpha}.
We define
\begin{equation}\label{3t31}
 p(z):=\frac{1+w(z)}{1-w(z)}=1+p_1z+p_2z^2+\cdots\quad(z\in\Delta).
\end{equation}
 It is clear that $p(0)=1$ and $p\in\mathcal{P}$. Relationships \eqref{sumB} and
\eqref{3t31} give us
\begin{equation*}
 1+\mathcal{B}_{\alpha}(w(z))=1+\frac{1}{2}A_1p_1z+\left(\frac{1}{4}A_2p_1^2
 +\frac{1}{2}A_1\left(p_2-\frac{1}{2}p_1^2\right)\right)z^2+\cdots,
\end{equation*}
where $A_1=1$ and $A_2=-\cos\alpha$. If we equate the coefficients of $z$ and $z^2$ on both sides of \eqref{2t31}, then we get
\begin{equation}\label{a2}
a_2=\frac{1}{2}p_1
\end{equation}
and
\begin{equation}\label{a3}
 a_3=\frac{1}{4}\left(p_2-\frac{1}{2}\cos\alpha p_1^2\right).
\end{equation}
From \eqref{bk+1 b2k+1}, \eqref{a2} and \eqref{a3}, we get
\begin{equation*}
 b_{k+1}=\frac{p_1}{2k},
\end{equation*}
and
\begin{equation*}
 b_{2k+1}=\frac{1}{4k}\left[p_2-\frac{1}{2}
 \left(\cos\alpha+\frac{k-1}{k}\right)p_1^2\right],
\end{equation*}
where $k\geq1$.
Therefore
\begin{equation*}
 b_{2k+1}-\mu
 b_{k+1}^2=\frac{1}{4k}\left[p_2-\frac{1}{2}\left(\cos\alpha+\frac{2\mu+k-1}{k}
 \right)p_1^2\right]\quad(\mu\in\mathbb{C}).
\end{equation*}
If we apply the Lemma \ref{lem. Fek-Sze} and letting
\begin{equation*}
 \mu':=\frac{1}{2}\left(\cos\alpha+\frac{2\mu+k-1}{k}
 \right),
\end{equation*}
then we get the desired inequality \eqref{1t31}.

From now, we shall show that the result is sharp. For the sharpness of the first and third cases of \eqref{1t31}, i.e. $\mu\leq \delta_1$ and $\mu\geq \delta_2$, respectively, consider the function
\begin{align*}
  f_1(z)&:=z\exp\left\{\int_{0}^{z}\frac{\mathcal{B}_\alpha(\xi)-1}{\xi}{\rm d}\xi\right\}\quad(z\in\Delta)\\
&=z+z^2+\frac{1}{2}(1-\cos\alpha)z^3+\frac{1}{18}(1-9\cos\alpha+8\cos^2\alpha)
z^4+\cdots,
\end{align*}
or one of its rotations.
It is easy to see that $f_1$ belongs to the class $\mathcal{M}(\alpha)$ and
\begin{equation*}
  (f_1(z^{k}))^{1/k}=z+\frac{1}{k}z^{k+1}
 +\left(\frac{1}{2k}(1-\cos\alpha)-\frac{1}{2}\frac{k-1}{k^2}\right)z^{2k+1}+\cdots.
\end{equation*}
The last equation shows that these inequalities are sharp. For the sharpness of the second inequality, we consider the function
\begin{align*}
  f_2(z):=z^2\exp\left\{\int_{0}^{z}\frac{\mathcal{B}_\alpha(\xi^2)-1}{\xi}{\rm d}\xi\right\}=z+\frac{1}{2}z^3+\cdots\quad(z\in\Delta).
\end{align*}
A simple calculation gives that
\begin{equation*}
  (f_2(z^{k}))^{1/k}=z+\frac{1}{2k}z^{2k+1}+\cdots.
\end{equation*}
Therefore the equality in the second inequality \eqref{1t31} holds for the $k$-th root transform of the above function $f_2$.
This completes the proof of Theorem \ref{t3.1}.
\end{proof}
The problem of finding sharp upper bounds for the coefficient
functional $|a_3-\mu a_2^2|$ for different subclasses of the
normalized analytic function class $\mathcal{A}$ is known as the
Fekete-Szeg\"{o} problem.
Therefore, if we let $k=1$ in the Theorem \ref{t3.1}, then we get the Fekete-Szeg\"{o} problem for the class $\mathcal{M}(\alpha)$ which we give in the following corollary.
\begin{corollary}\label{c3.1}
Let $\alpha\in[\pi/2,\pi)$ and $f\in\mathcal{M}(\alpha)$. Then for any
complex number $\mu$ we have
\begin{equation*}
 \left|a_3-\mu a_2^2\right|\leq\left\{
                                 \begin{array}{ll}
                                   \frac{1}{2}(1-\cos\alpha)-\mu, & \hbox{if\ \ $\mu\leq-\frac{1}{2}\cos\alpha$;} \\\\
                                   \frac{1}{2}, & \hbox{if\ \ $-\frac{1}{2}\cos\alpha\leq \mu\leq 1-\frac{1}{2}\cos\alpha$;} \\\\
                                   \frac{1}{2}(\cos\alpha-1)+\mu, & \hbox{if\ \ $\mu\geq1-\frac{1}{2}\cos\alpha$.}
                                 \end{array}
                               \right.
\end{equation*}
The result is sharp.
\end{corollary}
Putting $\alpha=\pi/2$ in the Corollary \ref{c3.1} we get the following.
\begin{corollary}\label{c3.2}
  Let the function $f$ be given by \eqref{f} satisfies the inequality
  \begin{equation*}
    \left|{\rm Re}\left\{\frac{zf'(z)}{f(z)}\right\}-1\right|<\frac{\pi}{4}\quad (z\in \Delta).
  \end{equation*}
  Then for any complex number $\mu\in\mathbb{C}$ we have the following sharp inequalities
\begin{equation*}
 \left|a_3-\mu
 a_2^2\right|\leq\left\{
                   \begin{array}{ll}
                     \frac{1}{2}-\mu, & \hbox{if\ \ $\mu\leq0$;} \\\\
                     \frac{1}{2}, & \hbox{if\ \ $0\leq\mu\leq1$;} \\\\
                     \mu-\frac{1}{2}, & \hbox{if\ \ $\mu\geq1$.}
                   \end{array}
                 \right.
\end{equation*}
\end{corollary}
If we let $\alpha\rightarrow \pi^{-}$ in the Corollary \ref{c3.1}, then we have:
\begin{corollary}
  If the function $f$ of the form \eqref{f} is starlike of order $1/2$, then for any complex number $\mu\in\mathbb{C}$ the following sharp inequalities hold true.
  \begin{equation*}
 \left|a_3-\mu a_2^2\right|\leq\left\{
                            \begin{array}{ll}
                            1-\mu, & \hbox{if\ \ $\mu\leq\frac{1}{2}$;} \\\\
                            \frac{1}{2}, & \hbox{$\frac{1}{2}\leq\mu\leq\frac{3}{2}$;} \\\\
                            \mu-1, & \hbox{if\ \ $\mu\geq\frac{3}{2}$.}
                            \end{array}
                            \right.
\end{equation*}
\end{corollary}
From \eqref{a2} and \eqref{a3} and the first case of the Lemma \ref{lem. Fek-Sze} we get.
\begin{corollary}
  If a function $f\in\mathcal{A}$ of the form $\eqref{f}$ belongs to the class $\mathcal{M}(\alpha)$ $(\pi/2\leq\alpha<\pi)$, then the following sharp inequalities hold.
\begin{equation*}
  |a_2|\leq1\quad{and}\quad|a_3|\leq \frac{1}{2}(1-\cos\alpha).
\end{equation*}
\end{corollary}
\section{Coefficient estimate and Fekete-Szeg\"{o} problem for the class $\mathcal{M}_\Sigma(\alpha)$}\label{sec3}
In this section, motivated by the Zaprawa's work (see \cite{Zaprawa}) we shall obtain the Fekete-Szeg\"{o} problem for the class $\mathcal{M}_\Sigma(\alpha)$. Also, we obtain upper bounds for the first coefficients $|a_2|$ and $|a_3|$ of the function $f$ of the form \eqref{f} belonging to the class $\mathcal{M}_\Sigma(\alpha)$. The coefficient estimate problem for each of the coefficients $|a_n|$ $(n\geq 4)$ is an open question.
Moreover, we apply the same technique as in \cite{ali2012}.
\begin{theorem}\label{t4.1}
  Let the function $f$ given by \eqref{f} be in the class $\mathcal{M}_\Sigma(\alpha)$ and $\pi/2\leq \alpha<\pi$. Then
\begin{equation}\label{estimate of a2}
  |a_2|\leq \sqrt{\frac{2}{2+\cos\alpha}}
\end{equation}
and for any real number $\mu$ we have
\begin{equation*}
  |a_3-\mu a_2^2|\leq
\left\{
  \begin{array}{ll}
    \frac{1}{2}, & \hbox{if\ \ $|1-\mu|\leq \frac{1}{2}\left(1+\frac{1}{2}\cos\alpha\right)$;} \\\\
    \frac{|1-\mu|}{1+\frac{1}{2}\cos\alpha}, & \hbox{if\ \ $|1-\mu|\geq \frac{1}{2}\left(1+\frac{1}{2}\cos\alpha\right)$.}
  \end{array}
\right.
\end{equation*}
\end{theorem}
\begin{proof}
  Let $f\in\mathcal{M}_\Sigma(\alpha)$ be of the form \eqref{f} and $g=f^{-1}$ be given by \eqref{f inverse}. Then by Definition \ref{Def M Singma alpha}, Lemma \ref{lemKES} and definition of subordination there exist two Schwarz functions $u:\Delta\rightarrow\Delta$ and $v:\Delta\rightarrow\Delta$ with the properties $u(0)=0=v(0)$, $|u(z)|<1$ and $|v(z)|<1$ such that
\begin{equation}\label{2sub1}
  \frac{zf'(z)}{f(z)}=1+\mathcal{B}_{\alpha}(u(z))\quad (z\in\Delta)
\end{equation}
and
\begin{equation}\label{2sub2}
\frac{wg'(w)}{g(w)}=1+\mathcal{B}_{\alpha}(v(z))\quad (z\in\Delta),
\end{equation}
where $\mathcal{B}_{\alpha}$ is defined by \eqref{B alpha}. Now we define the functions $k$ and $l$, respectively as follows
\begin{equation*}
  k(z)=\frac{1+u(z)}{1-u(z)}=1+k_1z+k_2z^2+\cdots\quad (z\in\Delta)
\end{equation*}
and
\begin{equation*}
 l(z)=\frac{1+v(z)}{1-v(z)}=1+l_1z+l_2z^2+\cdots\quad (z\in\Delta)
\end{equation*}
or equivalently
\begin{equation}\label{u}
  u(z)=\frac{k(z)-1}{k(z)+1}=\frac{1}{2}\left(k_1z+\left(k_2-\frac{1}{2}k_1^2\right)z^2
  +\cdots\right)
\end{equation}
and
\begin{equation}\label{v}
  v(z)=\frac{l(z)-1}{l(z)+1}=\frac{1}{2}\left(l_1z+\left(l_2-\frac{1}{2}l_1^2\right)z^2
  +\cdots\right).
\end{equation}
It is clear that the functions $k$ and $l$ belong to class $\mathcal{P}$ and $|k_i|\leq2$ and $|l_i|\leq2$ $(i=1,2,\ldots)$. 
From \eqref{sumB}, \eqref{2sub1}-\eqref{v}, we have
\begin{align}\label{1p4}
  \frac{zf'(z)}{f(z)}&=1+\mathcal{B}_{\alpha}\left(\frac{k(z)-1}{k(z)+1}\right)\\
  &=1+\frac{1}{2}A_1k_1z+\left(\frac{1}{2}A_1\left(k_2-\frac{1}{2}k_1^2\right)
  +\frac{1}{4}A_2k_1^2\right)z^2+\cdots\nonumber,
\end{align}
and
\begin{align}\label{2p4}
  \frac{wg'(w)}{g(w)}&=1+\mathcal{B}_{\alpha}\left(\frac{l(z)-1}{l(z)+1}\right)\\
  &=1+\frac{1}{2}A_1l_1z+\left(\frac{1}{2}A_1\left(l_2-\frac{1}{2}l_1^2\right)
  +\frac{1}{4}A_2l_1^2\right)z^2+\cdots\nonumber.
\end{align}
where $A_1=1$ and $A_2=-\cos\alpha$. Thus, upon comparing the corresponding coefficients in \eqref{1p4} and \eqref{2p4}, we obtain
\begin{equation}\label{a22}
  a_2=\frac{1}{2}A_1k_1=\frac{1}{2}k_1,
\end{equation}
\begin{equation}\label{a33}
 2a_3-a_2^2=\frac{1}{2}A_1\left(k_2-\frac{1}{2}k_1^2\right)+\frac{1}{4}A_2k_1^2
=\frac{1}{2}\left(k_2-\frac{1}{2}k_1^2\right)-\frac{k_1^2}{4}\cos\alpha,
 \end{equation}
 \begin{equation}\label{-a2}
   -a_2=\frac{1}{2}A_1l_1=\frac{1}{2}l_1,
\end{equation}
and
\begin{equation}\label{a333}
  3a_2^2-2a_3=\frac{1}{2}A_1\left(l_2-\frac{1}{2}l_1^2\right)+\frac{1}{4}A_2l_1^2
=\frac{1}{2}\left(l_2-\frac{1}{2}l_1^2\right)-\frac{l_1^2}{4}\cos\alpha.
\end{equation}
From equations \eqref{a22} and \eqref{-a2}, we can easily see that
\begin{equation}\label{k=-l}
  k_1=-l_1
\end{equation}
and
\begin{equation*}
  8a_2^2=(k_1^2+l_1^2).
\end{equation*}
If we add \eqref{a33} to \eqref{a333}, we get
\begin{equation}\label{2 a22= 1 2}
  2a_2^2=\frac{1}{2}\left[\left(k_2-\frac{1}{2}k_1^2\right)
  +\left(l_2-\frac{1}{2}l_1^2\right)\right]-\frac{1}{4}\cos\alpha\left(k_1^2+l_1^2\right).
\end{equation}
Substituting \eqref{a22}, \eqref{-a2} and \eqref{k=-l} into \eqref{2 a22= 1 2}, we obtain
\begin{equation}\label{k 1 pow 2}
  k_1^2=\frac{k_2+l_2}{2(1+(\cos\alpha)/2)}.
\end{equation}
Now, \eqref{a22} and \eqref{k 1 pow 2} imply that
\begin{equation}\label{a2 pow 2}
  a_2^2=\frac{k_2+l_2}{2(2+\cos\alpha)}.
\end{equation}
Since $|k_2|\leq2$ and $|l_2|\leq2$, \eqref{a2 pow 2} implies that
\begin{equation*}
  |a_2|\leq \sqrt{\frac{2}{2+\cos\alpha}},
\end{equation*}
which proves the first assertion \eqref{estimate of a2} of Theorem \ref{t4.1}.
Now, if we subtract \eqref{a333} from \eqref{a33} and use of \eqref{k=-l}, we get
\begin{equation}\label{a3=a2 pow 2}
  a_3=a_2^2+\frac{1}{8}(k_2-l_2).
\end{equation}
From \eqref{a2 pow 2} and \eqref{a3=a2 pow 2} it follows that
\begin{equation*}
  a_3-\mu a_2^2=\left(\frac{1}{8}+\hbar(\mu)\right)k_2
  +\left(\hbar(\mu)-\frac{1}{8}\right)l_2\quad(\mu\in\mathbb{R}),
\end{equation*}
where
\begin{equation*}
  \hbar(\mu):=\frac{1-\mu}{2(2+\cos\alpha)}\quad(\mu\in\mathbb{R}).
\end{equation*}
Since $|k_2|\leq2$ and $|l_2|\leq2$, we conclude that
\begin{equation*}
  |a_3-\mu a_2^2|\leq\left\{
                     \begin{array}{ll}
                      \frac{1}{2}, & \hbox{if\ \ $0\leq|\hbar(\mu)|\leq\frac{1}{8}$;} \\\\
                        4|\hbar(\mu)|, & \hbox{if\ \ $|\hbar(\mu)|\geq\frac{1}{8}$.}
                       \end{array}
                     \right.
\end{equation*}
This completes the proof.
\end{proof}
Taking $\mu=0$ in the above Theorem \ref{t4.1} we get.
\begin{corollary}
  Let $f$ of the form \eqref{f} be in the class $\mathcal{M}_\Sigma(\alpha)$. Then
\begin{equation*}
  |a_3|\leq \frac{1}{1+\frac{1}{2}\cos\alpha}\quad(\pi/2\leq\alpha<\pi).
\end{equation*}
\end{corollary}
If we let $\alpha\rightarrow \pi^-$ in the Theorem \ref{t4.1}, we get the following.
\begin{corollary}
  If the function $f$ of the form \eqref{f} belongs to the class $\mathcal{M}_\Sigma(1/2)$, then $|a_2|\leq 1$ and
\begin{equation*}
  |a_3-\mu a_2^2|\leq\left\{
                       \begin{array}{ll}
                         \frac{1}{2}, & \hbox{if \ \ $|1-\mu|\leq \frac{1}{4}$;} \\\\
                         2|1-\mu|, & \hbox{if \ \ $|1-\mu|\geq \frac{1}{4}$,}
                       \end{array}
                     \right.
\end{equation*}
where $\mu$ is real.
\end{corollary}


\end{document}